\documentclass[12 pt]{article}
\usepackage{amssymb,amsmath,amsfonts}
\usepackage{graphics}
\usepackage[english]{babel}
\usepackage{graphicx}
\usepackage{float}
\usepackage{sidecap}
\usepackage{tkz-graph}
\usetikzlibrary{patterns,arrows,decorations.pathreplacing}
\numberwithin{equation}{section}

\newenvironment{proof}{\noindent {\em {Proof}}.}{$\square$
\medskip}

\newtheorem{corollary}{Corollary}

\newtheorem{theorem}{Theorem}

\newtheorem{conjecture}{Conjecture}

\newtheorem{lm}{Lemma}

\begin{document}
\title{On the Signed Complete Graphs with Maximum Index}
\author{N. Kafai$^{a}$,~~F. Heydari$^{a}$,~~N. Jafari Rad$^{b}$,~~M. Maghasedi$^{a}$
}
\date{ }

\maketitle \vspace{-5mm}

\begin{center}
\date {}{ \small \textit{$^a$Department of Mathematics, Karaj Branch, Islamic Azad University, Karaj, Iran
\\
$^b$ Department of Mathematics, Shahed University, Tehran, Iran\footnote{{\it E-mail addresses}: navid.kafai@kiau.ac.ir (N. Kafai), f-heydari@kiau.ac.ir (F. Heydari), n.jafarirad@gmail.com (N.J. Rad), maghasedi@kiau.ac.ir (M. Maghasedi).}}
}
\end{center}

\begin{abstract}
Let $\Gamma=(K_{n},H^-)$ be a signed complete graph whose negative edges induce a subgraph $H$. The index of $\Gamma$ is the largest eigenvalue of its adjacency matrix. In this paper we study the index of $\Gamma$ when $H$ is a unicyclic graph. We show that among all signed complete graphs of order $n>5$ whose negative edges induce a unicyclic graph of order $k$ and maximizes the index, the negative edges induce a triangle with all remaining vertices being pendant at the same vertex of the triangle.
\end{abstract}

~~~\noindent\textbf{Keywords:} Signed complete graph, Unicyclic graph, Index. \\

~~~\noindent\textbf{MSC:} 05C22, 05C50, 15A18.

\section{Introduction}
~~~Let $G$ be a simple graph with the vertex set $V(G)$ and the edge set $E(G)$. The \textit {order} of $G$ is defined as $\vert V(G)\vert $.
The degree of a vertex $v$ in $G$ is denoted by $\deg_G(v).$ A vertex of degree one is called a \textit{pendant vertex}. We denote the set of all neighbors of $v$ in $G$ by $N_G(v)$.
Let $K_n$ be the \textit {complete graph} of order $n$
and $K_{1,k}$ denote the \textit {star graph} of order $k+1$.
A tree containing exactly two non-pendant vertices is called a \textit{double-star}. A double-star with degree sequence $ (s+1,t+1,1,\ldots,1) $ is denoted by $D_{s,t}$. By $C_k$ we denote a cycle of length $k$. A \emph{unicyclic} graph is a connected graph containing exactly one cycle. A {\it cactus} is a connected graph in which any two cycles have at most one common vertex.

A \textit{signed graph} $\Gamma$ is an ordered pair $(G,\sigma)$, where $G$ is a simple graph (called the \textit {underlying graph}), and $\sigma : E(G) \longrightarrow \lbrace -,+\rbrace$ is a mapping defined on the edge set of $G$ (called the \textit{signature}). If all edges of a signed graph $(G,\sigma)$ are positive (resp. negative), then we denote it by $(G,+)$ (resp. $(G,-)$).
For a subset $X\subseteq V(\Gamma)$, the subgraph induced by $X$ is denoted by $\Gamma[X]$.
Let $A(G)=(a_{ij})$ be the adjacency matrix of $G$. The \textit{adjacency matrix} of a signed graph $\Gamma=(G,\sigma)$ is defined as a square matrix $A(\Gamma)=(a_{ij}^\sigma)$, where $a_{ij}^\sigma=\sigma (v_iv_j)a_{ij}$.
The characteristic polynomial of a matrix $A$ is denoted by $\varphi(A,\lambda)$. The characteristic polynomial
of $A(\Gamma)$ is called the \emph{characteristic polynomial} of the signed graph $\Gamma$ and denoted by $\varphi(\Gamma,\lambda)$.
Also, the spectrum of $A(\Gamma)$ is called the spectrum of $\Gamma$ and the largest eigenvalue is often called the \textit{index}.
The spectrum of signed graphs has been studied by many authors, for instance see \cite{Akbari2, Belardo, sta1, sta2}.

Let $\Gamma=(G,\sigma)$ be a signed graph and $v \in V(\Gamma)$. We obtain a new graph $ \Gamma^\prime$ from $ \Gamma $ if we change the signs of all edges incident with $v$. We call $ v $ a \textit {switching vertex}. A \textit {switching} of a signed graph $\Gamma$ is a graph that can be obtained by applying finitely many switching operations. We call two graphs $\Gamma$ and $\Gamma^\prime$ \textit{switching equivalent} if $\Gamma^\prime$ is a switching of $\Gamma$ and we write $\Gamma \sim \Gamma^\prime$. The adjacency matrices of two switching equivalent signed graphs $\Gamma$ and $\Gamma^\prime$ are similar and hence they have the same spectrum, see \cite{Zaslavsky2}.

In \cite{Koled}, the connected signed graphs of fixed order, size, and number of negative edges with maximum index have been studied. It was conjectured in \cite{Koled} that if $ \Gamma $ is a signed complete graph of order $n$ with $k$ negative edges, $k<n-1$ and $\Gamma$ has maximum index, then negative edges induce the signed star $K_{1,k}.$ In \cite{Akbari3}, the authors proved that this conjecture holds for signed complete graphs whose negative edges form a tree. Recently, Ghorbani et al. \cite{gm} proved the conjecture. Let $(K_{n},H^-)$ denote a signed complete graph of order $n$ whose negative edges induce a subgraph $H$. They introduced a family of graphs $H_{m,n}$ for positive integers $n$ and $m$ with $m\leq \lfloor \frac{n^2}{4}\rfloor$ and proved that
among the signed complete graphs with $n$ vertices and $m$ negative edges,
$(K_n,H^-)$ has the maximum index if and only if $H$ is isomorphic to a $H_{n,m}$.

In this paper we focus on the signed complete graphs whose negative edges induce a unicyclic graph. We show that among all signed complete graphs of order $n>5$ whose negative edges induce a unicyclic graph of order $k$ and maximizes the index, the negative edges induce a triangle with all remaining vertices being pendant at the same vertex of the triangle. This result with a result of \cite{Akbari3} on trees lead to a conjecture on signed complete graphs whose negative edges induce a cactus graph.

\section{Preliminaries}
~~~The spectral theory of signed graphs has more varieties than unsigned graphs. But an important tool works in a similar way for signed graphs, which is a consequence of {\rm \cite[{Theorem} $1.3.11$]{Cvetkovic}}.

\begin{theorem}\rm{(Interlacing Theorem for signed graphs)}\label{Interlac}
Let $\Gamma$ be a signed graph with $n$ vertices and eigenvalues $\lambda _{1}\geq \cdots \geq \lambda_{n}$, and let $\Gamma^\prime$ be an induced subgraph of $\Gamma$ with $m$ vertices. If the eigenvalues of $\Gamma^\prime$ are $\mu _{1}\geq \cdots \geq \mu_{m}$, then $\lambda_{n-m+i} \leq \mu_i \leq \lambda_i$ for $i=1,\ldots,m.$
\end{theorem}

The following result will be useful in the sequel.

\begin{lm}\label{Kr,s} {\rm \cite[Lemma $3$]{Akbari3}}
Let $\Gamma=(K_{n},K_{1,k}^-)$ be a signed complete graph. Then
$$\varphi (\Gamma,\lambda)=(\lambda +1)^{n-3} \Big(\lambda ^3+(3-n)\lambda ^2 +(3-2n)\lambda +4k(n-k-1)+1-n\Big) .$$
\end{lm}

Also, we need to introduce an additional notation. Assume that $A$ is a symmetric real matrix of order $n$ and $\{X_1,\ldots,X_m\}$ is a partition of $[n] =\{1,\ldots,n\}.$ Let $\{X_1,\ldots,X_m\}$ partition the rows and columns of $A,$ as follows,
$$\begin{bmatrix} A_{1,1}&\cdots&A_{1,m} \\
\vdots&&\vdots\\
A_{m,1}&\cdots&A_{m,m}\end{bmatrix},$$
where $A_{i,j}$ denotes the submatrix of $A$ formed by rows in $X_i$ and the columns in $X_j.$ Then the $m \times m$ matrix $B= (b_{ij})$ is called the \textit{quotient matrix} related to that partition, where $b_{ij}$ denotes the average row sum of $A_{i,j}$. If the row sum of each $A_{i,j}$ is constant, then the partition is called \textit{equitable}, see \cite{ham}.

If $X=(x_1,\ldots,x_n)^T$ is an eigenvector corresponding to the eigenvalue $\lambda$ of a signed
graph $\Gamma=(G,\sigma)$, then we assume that the component $x_v$ corresponds to the vertex $v$. So the following is the {\it eigenvalue equation} for $v$: $$\lambda x_v=\sum_{u\in N_{G}(v)}\sigma(uv)x_u.$$

The next lemma is one of the most used tools in the identifications of graphs with maximum index.

\begin{lm}{\rm \cite[Lemma 5.1(i)]{Koled}}
\label{rotation}
Let $u$, $v$ and $w$ be distinct vertices of a signed graph $\Gamma$ and let $X=(x_1,\ldots,x_n)^T$ be an eigenvector corresponding to the index $\lambda_1(\Gamma)$. Let $\Gamma'$ be obtained
by reversing the sign of the positive edge $uv$ and the negative edge $uw$. If $x_u \geq 0$, $x_v \leq x_w$ or $x_u \leq 0$, $x_v\geq x_w$, then $\lambda_1(\Gamma) \leq \lambda_1(\Gamma')$. If at least one inequality is strict, then $\lambda_1(\Gamma) < \lambda_1(\Gamma')$.
\end{lm}

If $u$, $v$ and $w$ are the vertices given in Lemma \ref{rotation}, then $R(u, v, w)$ refers to the relocation described in the lemma.

\section{$(K_n,U^-)$ with maximum index}
One classical problem of graph spectra is to identify the maximal graphs with respect to the index in a given class of graphs. In this section, we determine $(K_n,U^-)$ with maximum index, where $U$ is a unicyclic subgraph of $K_n$ of order $k$. We begin with the following lemma.

\begin{lm}
\label{q1}
Let $\Gamma=(K_n,Q_1^-)$ be a signed complete graph with $k\geq 4$ negative edges, where $Q_1$ is the graph depicted in Fig. \ref{figure1}.
Then $$\varphi (\Gamma,\lambda)=(\lambda +1)^{n-5} \Big(\lambda^5 +(5-n)\lambda^4+(10-4n)\lambda ^3+(12k-6n+4ku-38)\lambda ^2 +$$
$$\ \ \ \ \ \ \ \ \ \ \ \ \ \ \  \ \ \ \ \ \ \ \ \ \ \ \ \ \ \ (24k-4n+8ku-91)\lambda +127n-116k-28ku-47\Big),$$
where $u=n-k$.
\end{lm}

\begin{proof}
{First assume that $k< n$. By switching $\Gamma$ at vertex $v_1$, one can deduce that $\Gamma=(K_n,Q_1^-)$ is switching equivalent to $(K_n, D_{n-k,2}^-)$. Thus the characteristic polynomials of $\Gamma=(K_n,Q_1^-)$ and $(K_n, D_{n-k,2}^-)$ are the same. Hence by \cite[Theorem $4$]{Akbari3} and \cite[Remark $5$]{Akbari3}, the result holds. Now, assume that $k=n$. Again by switching $\Gamma$ at vertex $v_1$, we conclude that $\Gamma=(K_n,Q_1^-)$ is switching equivalent to $(K_n, K_{1,3}^-)$. Therefore, by Lemma \ref{Kr,s},
$$\varphi (\Gamma,\lambda)=(\lambda +1)^{n-3} \Big(\lambda ^3+(3-n)\lambda ^2 +(3-2n)\lambda +11n-47\Big)=$$
$$(\lambda +1)^{n-5} \Big(\lambda^5 +(5-n)\lambda^4+(10-4n)\lambda ^3+(6n-38)\lambda ^2 +(20n-91)\lambda +11n-47\Big).$$
Hence the proof is complete.}
\end{proof}

\begin{figure}[h!]
\centering
		\begin{minipage}[b]{.4\textwidth}
		\centering
		\begin{tikzpicture}[scale=1.2, baseline={(0, -1.15)}]
			\draw (0,0)--(-1,0)--(-1,1)--(0,1);
			\draw (0,0)--(0,1);
			\begin{scope}[shift={(0,1)}]
				\draw \foreach \x in {50,130}{(\x:1)--(0:0)};
				\fill[black] \foreach \x in {50,130} {(\x:1) circle (0.05)};
				\node at (90:1.2) {\small $k-4$};
                \draw [decorate,decoration={brace,amplitude=5pt,mirror},xshift=0pt,yshift=0pt](.7,.85) -- (-.7,.85) node[black,midway,yshift=-.3cm] {\footnotesize $ $};
			\end{scope}
            \draw [dotted] (.6, 1.77)-- (-.6,1.77);
			\fill[black] (0,1) circle (0.05) node[right]{\small $v_1$};
			\fill[black] (0,0) circle (0.05) node[right]{\small $v_2$};
			\fill[black] (-1,0) circle (0.05) node[left]{\small $v_3$};
			\fill[black] (-1,1) circle (0.05) node[left]{\small $v_4$};
			\node at (0.3,-0.7) {${{Q}_{1}}\qquad\qquad$};
		\end{tikzpicture}
		\end{minipage}
		\begin{minipage}[b]{.4\textwidth}
		\centering
		\begin{tikzpicture}[scale=1.2]
			\draw (0,0)--(-1,0)--(-1,1)--(0,1);
			\draw (0,0)--(0,1);
            \draw [decorate,decoration={brace,amplitude=5pt,mirror},xshift=0pt,yshift=0pt](.4,2) -- (-.4,2) node[black,midway,yshift=-.3cm] {\footnotesize $ $};
			\begin{scope}[shift={(0,1)}]
				\draw \foreach \x in {70,110}{(\x:1)--(0:0)};
				\fill[black] \foreach \x in {70,110} {(\x:1) circle (0.05)};
				\node at (90:1.3) {\small $s$};
			\end{scope}
            \draw [dotted] (.3, 1.95)-- (-.3, 1.95);
            \draw (0,0)--(1,0);
            \draw (0,0)--(1,-.5);
            \draw [decorate,decoration={brace,amplitude=5pt,mirror},xshift=0pt,yshift=0pt](1.05,-.55) -- (1.05,.05) node[black,midway,yshift=-.3cm] {\footnotesize $ $};
            \fill[black] (1,0) circle (0.05) node[right]{\small $ $};
            \fill[black] (1,-.5) circle (0.05) node[right]{\small $ $};
            \draw [dotted] (1, 0)-- (1, -.5);
			\node at (-10:1.4) {\small $t$};
			\fill[black] (0,1) circle (0.05) node[right]{\small $v_1$};
			\fill[black] (0,0) circle (0.05) node[above right]{\small $v_2$};
			\fill[black] (-1,0) circle (0.05) node[left]{\small $v_3$};
			\fill[black] (-1,1) circle (0.05) node[left]{\small $v_4$};
			\node at (1.5,0.6) {\small $k=s+t+4$};
			\node at (0.5,-0.7) {${{Q(s,t)}}\qquad\qquad$};
		\end{tikzpicture}
		\end{minipage}
\caption{The unicyclic graphs $Q_1$, $Q(s,t)$.}
\label{figure1}
	\end{figure}
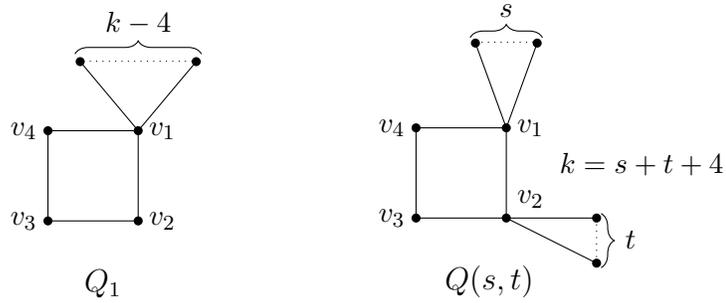

Let $\Gamma =(K_{n},\sigma)$ be a signed complete graph. Let $\sqcap=\{X_1,\ldots,X_p,
X_{p+1},\ldots,$
$X_{p+q}\}$ be a partition of $V(\Gamma)$ such that
all edges between $X_i$ and $X_j$ have the same sign for each $i, j$, $\Gamma[X_i]=(K_{n_i},+)$ for $i=1,\ldots,p$, and $\Gamma[X_i]=(K_{n_i},-)$ for $i=p+1,\ldots,p+q$, where $|X_i|=n_i$ for $i=1,\ldots,p+q$.
Obviously, $\sqcap$ is an equitable partition of $V(\Gamma)$.
Moreover, we have the next theorem.

\begin{theorem}\label{charpoly}
Let $\Gamma =(K_{n},\sigma)$ be a signed complete graph. Let $\sqcap=\{X_1,\ldots,X_p,$
$X_{p+1},\ldots,X_{p+q}\}$ be a partition of $V(\Gamma)$ with the above properties and $B$ be the quotient matrix of $A(\Gamma)$ related to $\sqcap$.
If $m_1=\sum _{i=1} ^{p}n_i$ and $m_2=\sum _{i=p+1} ^{p+q}n_i$, then
$$\varphi (\Gamma,\lambda)=(\lambda +1)^{m_1-p}(\lambda -1)^{m_2-q}\varphi (B,\lambda).$$
\end{theorem}

\begin{proof}
{Suppose that $X_1=\{v_1,\ldots,v_{n_1}\}$ and $X_i=\{v_{n_1+\cdots+n_{i-1}+1},\ldots,v_{n_1+\cdots+n_{i}}\}$, for $i=2,\ldots,p+q$. Then we have,

$$\lambda I-A(K_{n},\sigma)= \left[\begin{tabular}{c | c}
 $\lambda I-A(K_{n_1},+)$ &  $\ast$ \\ \hline
      $\ast$  &   $\ast$
\end{tabular}\right].$$


We apply finitely many elementary row and column operations on the matrix $\lambda I-A(K_n,\sigma)$. First, subtract the $n_1$th row from all the upper rows. This leads to the following matrix,

$$\left[\begin{tabular}{c c c c | c}
 $\lambda+1$  &  &   &  $-\lambda-1$ & \\
 $\bf 0$ & $\ddots$  & $\bf 0$ & $\vdots$ & $\bf 0$ \\
     &  &  $\lambda+1$  & $-\lambda-1$  &  \\
 & $\bf -1$  &  &  $\lambda$ & $\ast$ \\ \hline
  & & $\ast$ & & $\ast$
\end{tabular}\right].$$

Now, adding the first $n_1-1$ columns to $n_1$th column, we
obtain the following matrix,

$$\left[\begin{tabular}{c c c c | c}
 $\lambda+1$  &  &   &  $0$ & \\
 $\bf 0$ & $\ddots$  & $\bf 0$ & $\vdots$ & $\bf 0$ \\
     &  &  $\lambda+1$  & $0$  &  \\
 & $\bf -1$  &  &  $\lambda+1-n_1$ & $\ast$ \\ \hline
  & & $\ast$ & & $\ast$
\end{tabular}\right].$$

Thus the following holds:
$$\varphi (\Gamma,\lambda)=(\lambda +1)^{n_1-1}\det \left[\begin{tabular}{c | c}
 $\lambda+1-n_1$ &  $\ast$ \\ \hline
      $\ast$  &   $\ast$
\end{tabular}\right].$$

By repeating the same procedure, one can deduce that
$$\varphi (\Gamma,\lambda)=(\lambda +1)^{m_1-p}(\lambda -1)^{m_2-q}\varphi (B,\lambda),$$
as desired.
}
\end{proof}

\begin{lm}
\label{qst}
Let $\Gamma=(K_n,Q(s,t)^-)$, $s,t\geq 1$, be a signed complete graph with $k$ negative edges, where $Q(s,t)$ is the graph depicted in Fig. \ref{figure1}.
Then $$\varphi (\Gamma,\lambda)=(\lambda +1)^{n-7} \Big(\lambda^7 +(7-n)\lambda^6+(21-6n)\lambda ^5+(12k-15n+4ku+8st-13)\lambda ^4 +$$
$$(48k-20n+16ku+32st-157)\lambda^3 +(113n-56k-8ku-16st(u-1)-267)\lambda^2+$$
$$(250n-208k-48ku-32st(u+1)-185)\lambda+127n-116k-28ku+24st(2u-1)-47\Big),$$
where $u=n-k$.
\end{lm}

\begin{proof}
{Assume that $k< n$. Suppose that $X_1=\{v_1\}$, $X_2=\{v_2\}$, $X_3=\{v_3\}$, $X_4=\{v_4\}$, $X_5=N_{Q(s,t)}(v_1)\setminus \{v_2, v_4\}$, $X_6=N_{Q(s,t)}(v_2)\setminus \{v_1, v_3\}$, and $X_7=V(K_n)\setminus V(Q(s,t))$, see Fig. \ref{figure1}.
Let $B$ be the quotient matrix of $A(\Gamma)$ related to $\{X_1,\ldots,X_7\}$.
It is not hard to see that the characteristic polynomial of $B$ is
$$\varphi (B,\lambda)= \lambda^7 +(7-n)\lambda^6+(21-6n)\lambda^5+(12k-15n+4ku+8st-13)\lambda ^4 +$$
$$(48k-20n+16ku+32st-157)\lambda^3 +(113n-56k-8ku-16st(u-1)-267)\lambda^2+$$
$$(250n-208k-48ku-32st(u+1)-185)\lambda+127n-116k-28ku+24st(2u-1)-47,$$
where $u=n-k$. By Theorem \ref{charpoly}, we have
$\varphi (\Gamma,\lambda)=(\lambda +1)^{n-7}\varphi (B,\lambda).$
If $k=n$, let $X_1,\ldots,X_6$ be as above and $B'$ be the quotient matrix of $A(\Gamma)$ related to $\{X_1,\ldots,X_6\}$. Then
$\varphi (\Gamma,\lambda)=(\lambda +1)^{n-6}\varphi (B',\lambda).$ So one can deduce that
$$\varphi (\Gamma,\lambda)=(\lambda +1)^{n-6}\Big(\lambda^6 +(6-n)\lambda^5+(15-5n)\lambda^4+(2n+8st-28)\lambda ^3 +$$
$$\ \ \ \ \ \ \ \ \ \ \ \ \ \ \ \ (26n+24st-129)\lambda^2 +(31n-8st-138)\lambda+
11n-24st-47\Big),$$
and hence
$$\varphi (\Gamma,\lambda)=(\lambda +1)^{n-7}\Big(\lambda^7 +(7-n)\lambda^6+(21-6n)\lambda^5+(-3n+8st-13)\lambda ^4 +$$
$$(28n+32st-157)\lambda^3 +(57n+16st-267)\lambda^2+
(42n-32st-185)\lambda+11n-24st-47\Big).$$
This completes the proof.}
\end{proof}

\begin{corollary}\label{q1st}
Let $\Gamma=(K_n,Q(s,t)^-)$, $s,t\geq 1$, and $\Gamma'=(K_n,Q_1^-)$ be two signed complete graphs with $k$ negative edges, where $Q(s,t)$ and $Q_1$ are the graphs depicted in Fig. \ref{figure1}. Then $\lambda_1(\Gamma) <\lambda_1(\Gamma^\prime).$
\end{corollary}

\begin{proof}
{Let $\lambda_1=\lambda_1(\Gamma)$. By Lemmas \ref{q1} and \ref{qst}, we deduce that
$$\varphi(\Gamma^\prime,\lambda) -\varphi(\Gamma ,\lambda) =-8st(\lambda+1)^{n-7}\big(\lambda^4+4\lambda^3-2(n-k-1)\lambda^2-4(n-k+1)\lambda+$$
$$\ \ \ \ \ \ \ \ \ \  \ \ \ \ \ \ \ \ \ \ \ \ \ \ \   3(2n-2k-1)\big).$$
Setting $\lambda=\lambda_1$, we have
$$\varphi(\Gamma^\prime,\lambda_1)=-8st(\lambda_1+1)^{n-7}\big(\lambda_1^4+4\lambda_1^3-2(n-k-1)\lambda_1^2
-4(n-k+1)\lambda_1+3(2n-2k-1)\big).$$
Since $\Gamma$ has $(K_{n-3}, + )$ as an induced subgraph, by Theorem \ref{Interlac}, we conclude that
$n-k+1<n-4\leq \lambda_1.$ Therefore, $\varphi(\Gamma^\prime,\lambda_1(\Gamma))<0$ which implies that $\lambda_1(\Gamma) <\lambda_1(\Gamma^\prime).$
}
\end{proof}

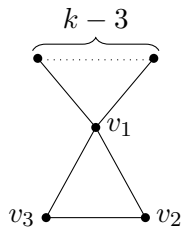
\begin{figure}[h!]
\centering
		\begin{minipage}[b]{.4\textwidth}
		\centering
		\begin{tikzpicture}[scale=1.2, baseline={(0, -1.15)}]
			\draw (0.55,0)--(-0.55,0)--(0,1);
			\draw (0.55,0)--(0,1);
			\begin{scope}[shift={(0,1)}]
				\draw \foreach \x in {50,130}{(\x:1)--(0:0)};
				\fill[black] \foreach \x in {50,130} {(\x:1) circle (0.05)};
				\node at (90:1.2) {\small $k-3$};
            \draw [decorate,decoration={brace,amplitude=5pt,mirror},xshift=0pt,yshift=0pt](.7,.85) -- (-.7,.85) node[black,midway,yshift=-.3cm] {\footnotesize $ $};
			\end{scope}
            \draw [dotted] (.6, 1.75)-- (-.6, 1.75);
			\fill[black] (0.55,0) circle (0.05) node[right]{\small $v_2$};
			\fill[black] (0,1) circle (0.05) node[right]{\small $v_1$};
			\fill[black] (-0.55,0) circle (0.05) node[left]{\small $v_3$};
		\end{tikzpicture}
		\end{minipage}
\caption{The unicyclic graph $U_1$.}
\label{figure2}
	\end{figure}

By a proof similar to the proof of Lemma \ref{qst}, we have the following lemma.

\begin{lm}
\label{u1}
Let $\Gamma=(K_n,U_1^-)$ be a signed complete graph with $k\geq 3$ negative edges, where $U_1$ is the graph depicted in Fig. \ref{figure2}.
Then $$\varphi (\Gamma,\lambda)=(\lambda +1)^{n-5} (\lambda -1) \Big(\lambda^4 +(6-n)\lambda^3+(16-5n)\lambda ^2+(4k-11n+4ku+18)\lambda +$$
$$\ \ \ \ \ \ \ \ \ \ \ \ \ \ \ \ \ \ \ \ \ \ \ \ \ \ \ \ \ \ \ 28k-31n+12ku+7\Big),$$
where $u=n-k$.
\end{lm}

Now, we can state the main result of this paper.

\begin{theorem}\label{theoremX}
 Among all signed complete graphs of order $n>5$ whose negative edges induce a unicyclic graph of order $k$ and maximizes the index, the negative edges induce a triangle with all remaining vertices being pendant at the same vertex of the triangle.
\end{theorem}

\begin{proof}
{Suppose that $k$ negative edges induce a unicyclic graph $U$ and maximizes the index, where $k\leq n$. Let $\Gamma=(K_{n},U^-)$ and $\lambda_1=\lambda_1(\Gamma)$. Clearly,
$(K_{n-k+1}, +)$ is an induced subgraph of $\Gamma$. Hence by Theorem \ref{Interlac}, we deduce that $n-k\leq \lambda_1$.
Suppose $U$ consists of the cycle, say $C$, of length $g$ and a certain number of trees attached at vertices of $C$. Let $V(U)=\{v_1,\ldots,v_k\}$ and $V(C)=\{v_1,\ldots,v_g\}$. Let $X=(x_1,\ldots,x_n)^T$ be an eigenvector corresponding to $\lambda_1$.

We claim that there is an integer $i$, $1 \leq i \leq g$, such that $x_i\neq 0$. To see this, by contrary assume that $x_i=0$, for $i=1,\ldots,g$. Let $v_p$ and $v_{q}$ be two consecutive vertices of $C$, that is $v_pv_q\in E(C)$, and suppose that two vertices $v_j$ ($g<j\leq k$) and $v_p$ are adjacent in $U$.
If $x_j\neq 0$, then by Lemma \ref{rotation}, the relocation $R(v_j,v_q,v_p)$ would contradict the maximality of $\lambda_1.$ By the same argument, one can prove that $x_i=0$ for $i=1,\ldots,k$. Thus $k< n$. Let $x_t$ be an entry of $X$ having the
largest absolute value. We may assume that $x_t>0$ (otherwise, consider $-X$ instead of $X$). By the eigenvalue equation for $v_t$, we have $\sum_{i=k+1, i\neq t}^{n}x_i=\lambda_1x_t$ which implies that $\lambda_1\leq n-k-1$, a contradiction. The claim is proved.

Suppose that $g>3$. Without loss of generality, assume that $x_1>0$. If $x_{g-1} \leq x_2$, then the possibility of $R(v_1,v_{g-1},v_2)$ contradicts the maximality of $\lambda_1$. So $x_2<x_{g-1}$.
Also, $x_g<x_3$ since otherwise, the relocation $R(v_1,v_3,v_g)$ leads to a contradiction. Now, if $x_2\geq 0$ (resp. $x_g\geq 0$), then the relocation $R(v_2,v_g,v_3)$ (resp. $R(v_g,v_2,v_{g-1})$) contradicts the maximality of $\lambda_1$. Hence, $x_2, x_g<0$. Thus for $g\geq 5$, if $x_{g-1}\geq x_3$ (resp. $x_3\geq x_{g-1}$), then by $R(v_2,v_{g-1},v_3)$ (resp. $R(v_g,v_3,v_{g-1})$), we get a contradiction. Therefore, $g=4$.

From the previous statement, we have $x_1>0$ and $x_2, x_4<0$. Also, $x_3>0$ since otherwise, the relocation $R(v_3,v_1,v_2)$ concludes a contradiction. Suppose that $v_1v_p, v_3v_q\in E(U)$, where $p,q>4$. If $x_p\geq 0$ (resp. $x_q\geq 0$), then the relocation $R(v_p,v_2,v_1)$ (resp. $R(v_q,v_4,v_3)$) contradicts the maximality of $\lambda_1$. Thus $x_p, x_q<0$. So if $x_3\geq x_1$ (resp. $x_1\geq x_3$), then by $R(v_p,v_3,v_1)$ (resp. $R(v_q,v_1,v_3)$), we find a contradiction. It follows that $\deg_U (v_1)=2$ or $\deg_U (v_3)=2$. Similarly, if $v_2v_p, v_4v_q\in E(U)$ and $p,q>4$, then we have a contradiction and hence $\deg_U (v_2)=2$ or $\deg_U (v_4)=2$. Without loss of generality, assume that $\deg_U (v_3)=\deg_U (v_4)=2$.

Now, we prove that the trees attached to vertices $v_1, v_2$ in $U$, if any, are stars.
For proving this, first assume that $v_1v_p, v_pv_q\in E(U)$, where $p, q>4$. If $x_p\geq 0$, then the relocation
$R(v_p,v_2,v_1)$ gives a contradiction and hence $x_p<0$. If $x_q\leq 0$, then the relocation $R(v_q,v_1,v_p)$ leads to a contradiction. Thus $x_q>0$. Therefore, if $x_2\leq x_p$ (resp. $x_p\leq x_2$), then by $R(v_q,v_2,v_p)$ (resp. $R(v_3,v_p,v_2)$), we get a contradiction. This implies that $v_p$ is a pendant vertex of $U$. By repeating the same procedure one can easily prove that the tree which is attached to the vertex $v_2$ in $U$, if any, is a star. So we conclude that $U$ may be equal to the graphs $Q_1$ or $Q(s,t)$, depicted in Fig. \ref{figure1}. By Corollary \ref{q1st}, $\lambda_1(K_n,Q(s,t)^-)<\lambda_1(K_n,Q_1^-)$ and hence $U\neq Q(s,t)$.

Next, suppose that $g=3$. We first prove that at least two vertices of the cycle $C$ have degree $2$ in $U$. By contrary assume that $v_1v_p, v_2v_q\in E(U)$, where $p,q>3$. If $x_1=x_2=0$, then $R(v_p,v_2,v_1)$ implies that $x_p=0$. Hence by $R(v_3,v_p,v_1)$, we deduce that $x_3=0$, a contradiction. Without loss of generality, assume that $x_1>0$. If $x_q\leq x_2$ (resp. $x_q\leq x_3$), then the relocation $R(v_1,v_q,v_2)$ (resp. $R(v_1,v_q,v_3)$) contradicts the maximality of $\lambda_1$. Thus $x_q>x_2, x_3$. Therefore, $x_3>0$ (by $R(v_3,v_q,v_2)$). If $x_1\leq x_2$, then $R(v_q,v_1,v_2)$ yields a contradiction. So $x_2<x_1$ and hence $x_p<0$ (by $R(v_p,v_2,v_1)$). Now, $R(v_3,v_p,v_1)$ gives the final contradiction. This completes the assertion.

We now prove that the tree attached to a vertex of $C$ in $U$, if any, is a star. By contrary assume that $v_1v_p, v_pv_q\in E(U)$, where $p,q>3$. Without loss of generality, assume that $x_q\geq 0$ (otherwise, consider $-X$ instead of $X$). If $x_1<x_p$, then the relocation $R(v_q,v_1,v_p)$ contradicts the maximality of $\lambda_1$. Hence $x_p\leq x_1$. Thus the relocations $R(v_2,v_p,v_1)$ and $R(v_3,v_p,v_1)$, respectively,
imply that $x_2\leq 0$ and $x_3\leq 0$. Therefore, by $R(v_2,v_q,v_3)$, we deduce that $x_2=x_3=x_q=0$. Finally, the relocation $R(v_1,v_q,v_2)$ yields $x_1=0$ which is impossible, and we are done.

So the candidates for maximizers are: $\Gamma=(K_{n},U_1^-)$, and $\Gamma'=(K_{n},Q_1^-)$, see Figs. \ref{figure1} and \ref{figure2}. We just have to compare the indices. By Lemmas \ref{q1} and \ref{u1}, we find that
$$\varphi(\Gamma,\lambda) -\varphi(\Gamma^\prime ,\lambda) =-8(\lambda+1)^{n-5}\big((k-5)\lambda^2+2(n-5)\lambda+12n-11k-2k(n-k)-5\big).$$
Let $\lambda_1=\lambda_1(\Gamma')$. Since $\Gamma'$ has $(K_{n-2}, + )$ as an induced subgraph, by Theorem \ref{Interlac}, we conclude that $\lambda_1\geq n-3$. Setting $\lambda=\lambda_1$, we have
$$\varphi(\Gamma,\lambda_1)=-8(\lambda_1+1)^{n-5}\big((k-5)\lambda_1^2+2(n-5)\lambda_1+12n-11k-2k(n-k)-5\big).$$
First assume that $k\geq 5$. Thus $(k-5)\lambda_1^2+2(n-5)\lambda_1+12n\geq (k-5)(n-3)^2+2(n-5)(n-3)+12n=kn^2+26n+9k-3n^2-6kn-15$. It is not hard to see that $kn^2+2k^2+26n>3n^2+8kn+2k+20$ which implies that $kn^2+26n+9k-3n^2-6kn-15>11k+2k(n-k)+5$. Hence $\varphi(\Gamma,\lambda_1)<0$ and so $\lambda_1(\Gamma')<\lambda_1(\Gamma)$. This is exactly what we need here. Note that if $n=k=5$, by switching $\Gamma'$ at vertex $v_3$, we have $\Gamma' \sim \Gamma$ and hence $\lambda_1(\Gamma')=\lambda_1(\Gamma)$.
Finally, assume that $k=4$. Thus
 $$\varphi(\Gamma,\lambda_1)=8(\lambda_1+1)^{n-5}\big(\lambda_1^2-2(n-5)\lambda_1-4n+17\big).$$
The roots of $\lambda^2-2(n-5)\lambda-4n+17$ are $n-5\pm \sqrt{n^2-6n+8}$.
If $n>7$, then $n-5-\sqrt{n^2-6n+8}< n-3\leq \lambda_1\leq n-1<n-5+\sqrt{n^2-6n+8}$, so $\varphi(\Gamma,\lambda_1)<0$ which yields that $\lambda_1(\Gamma')<\lambda_1(\Gamma)$. By a computer search, one can see that
the result holds for $n=6,7$, but $\lambda_1(K_{4},U_1^-)<\lambda_1(K_{4},C_4^-)$ and $\lambda_1(K_{5},U_1^-)<\lambda_1(K_{5},C_4^-)$.
The proof is complete.}
\end{proof}

The following is a direct consequence of Theorem \ref{theoremX} which confirms the  Koledin-Stani\'c conjecture for signed
complete graphs whose negative edges induce a unicyclic graph.

\begin{corollary}\label{u1st}
Let $(K_n,U^-)$ be a signed complete graph whose negative edges induce a unicyclic graph $U$ of order $k<n$. Then $\lambda_1(K_n,U^-) <\lambda_1(K_n,K_{1,k}^-).$
\end{corollary}

\begin{proof}
{Suppose that $\Gamma=(K_n,K_{1,k}^-)$ and $\Gamma'=(K_n,U_1^-)$, where $U_1$ is the graph depicted in Fig. \ref{figure2}. Let $\lambda_1=\lambda_1(\Gamma')$. Since $\Gamma'$ has $(K_{n-2}, + )$ as an induced subgraph, by Theorem \ref{Interlac}, we conclude that $\lambda_1\geq n-3$. By Lemmas \ref{Kr,s} and \ref{u1}, we deduce that
$$\varphi(\Gamma,\lambda) -\varphi(\Gamma' ,\lambda) =-8(\lambda+1)^{n-5}\big((k-1)\lambda^2-2(n-2k+1)\lambda+$$
$$\ \ \ \ \ \ \ \ \ \ \ \ \ \ \ \ \ \ \ \ \ \ \ \ \ \ \ \ \ \ \ \ \ \ \ \ \ \ \ \ \ \ \ \ \ \ \ \ \ \ \ \ \ \ \ \ \ \ \ \ \ \ \ \ \ \ \ \ 4n-3k-2k(n-k)-1\big).$$
Setting $\lambda=\lambda_1$, we have
$$\varphi(\Gamma,\lambda_1)=-8(\lambda_1+1)^{n-5}\big((k-1)\lambda_1^2-2(n-2k+1)\lambda_1+$$
$$\ \ \ \ \ \ \ \ \ \ \ \ \ \ \ \ \ \ \ \ \ \ \ \ \ \ \ \ \ \ \ \ \ \ \ \ \ \ \ \ \ \ \ \ \ \ \ \ \ \ \ \ \ \ \ \ \ \ 4n-3k-2k(n-k)-1\big).$$
If $k=3$, then $\varphi(\Gamma,\lambda_1)=-16(\lambda_1+1)^{n-4}\big(\lambda_1-n+4\big)$.
Thus $\varphi(\Gamma,\lambda_1)<0$ and hence $\lambda_1(\Gamma') <\lambda_1(\Gamma).$
Note that $n\geq 4$.

Now, assume that $k>3$. By a computer search, one can see that $\lambda_1(K_{5},U_1^-)<\lambda_1(K_{5},C_4^-)<\lambda_1(K_{5},K_{1,4}^-)$, so
the result holds for $n=5$. In what follows, we consider two cases.

$\bold{Case}$ 1: $n>2k-1$. We claim that $\varphi(\Gamma,\lambda_1)<0$. We have $(k-1)\lambda_1^2+4n\geq (k-1)(n-3)^2+4n=kn^2+10n+9k-n^2-6kn-9$. Since  $\lambda_1\leq n-1$, hence $2(n-2k+1)\lambda_1+3k+2k(n-k)+1 \leq 2(n-2k+1)(n-1)+3k+2k(n-k)+1=2n^2+7k-2k^2-2kn-1$. It is easy to see that $kn^2+2k^2+10n+2k>3n^2+4kn+8$ which yields that $kn^2+10n+9k-n^2-6kn-9>2n^2+7k-2k^2-2kn-1$. The claim is proved. Therefore, $\lambda_1(\Gamma') <\lambda_1(\Gamma).$

$\bold{Case}$ 2: $n\leq 2k-1$. Then $(k-1)\lambda_1^2-2(n-2k+1)\lambda_1+4n\geq (k-1)(n-3)^2-2(n-2k+1)(n-3)+4n$. It is not hard to check that $kn^2+2k^2+14n>3n^2+6k+4kn+4$ which implies that $(k-1)(n-3)^2-2(n-2k+1)(n-3)+4n>3k+2k(n-k)+1$. Thus $\varphi(\Gamma,\lambda_1)<0$ and hence $\lambda_1(\Gamma') <\lambda_1(\Gamma).$

Therefore, by Theorem \ref{theoremX}, the proof is complete.
}
\end{proof}

\begin{figure}[h!]
\centering
		\begin{minipage}[b]{.4\textwidth}
		\centering
		\begin{tikzpicture}[scale=1.2, baseline={(0, -1.15)}]
			\draw (0,1)--(-1,0)--(-0.55,0)--(0,1)--(0.55,0)--(1,0)--(0,1);
			\draw [dotted] (-0.55,0)--(0.55,0);
			\begin{scope}[shift={(0,1)}]
				\draw \foreach \x in {50,130}{(\x:1)--(0:0)};
				\fill[black] \foreach \x in {50,130} {(\x:1) circle (0.05)};
				\node at (90:1.2) {\small $k-3t$};
            \draw [decorate,decoration={brace,amplitude=5pt,mirror},xshift=0pt,yshift=0pt](.7,.85) -- (-.7,.85) node[black,midway,yshift=-.3cm]
            {\footnotesize $ $};
            \draw [decorate,decoration={brace,amplitude=5pt,mirror},xshift=0pt,yshift=0pt](-1.1,-1.05) -- (1.1,-1.05) node[black,midway,yshift=-.3cm]
            {\footnotesize $ $};
            \node at (0,-1.4) {\small $2t$};
			\end{scope}
            \draw [dotted] (.6, 1.75)-- (-.6, 1.75);
			\fill[black] (0.55,0) circle (0.05) node[right]{\small $ $};
			\fill[black] (0,1) circle (0.05) node[right]{\small $ $};
			\fill[black] (-0.55,0) circle (0.05) node[left]{\small $ $};
            \fill[black] (-1,0) circle (0.05) node[right]{\small $ $};
			\fill[black] (1,0) circle (0.05) node[left]{\small $ $};
		\end{tikzpicture}
		\end{minipage}
\caption{The cactus graph $G_t$.}
\label{figure3}
	\end{figure}
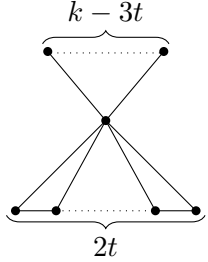

Let $G_t$ be the cactus graph with $k$ edges and $t$ cycles, depicted in Fig. \ref{figure3} ($G_0$ is the star graph $K_{1,k}$). We close this paper with the following conjecture.

\begin{conjecture}\label{con2}
 Among all signed complete graphs of order $n>5$ whose negative edges induce a cactus graph with $k$ edges and $t$ cycles ($k-t<n$), and maximizes the index, negative edges induce the graph $G_t$.
\end{conjecture}

We note that Conjecture \ref{con2} is true for $t=0,1$, according to Theorem \ref{theoremX} and a result of \cite{Akbari3}.



\small
\begin{thebibliography}{99}



\bibitem{Akbari2}S. Akbari, S. Dalvandi, F. Heydari, M. Maghasedi, On the eigenvalues of signed complete graphs, Linear Multilinear Algebra 67(3) (2019) 433--441.

\bibitem{Akbari3}S. Akbari, S. Dalvandi, F. Heydari, M. Maghasedi, Signed complete graphs with maximum index, Discuss. Math. Graph Theory 40 (2020) 393-403.

\bibitem{Belardo}F. Belardo, P. Petecki, Spectral characterizations of signed lollipop graphs, Linear Algebra Appl. 480 (2015) 144--167.

\bibitem{ham}A.E. Brouwer, W.H. Haemers, Spectra of Graphs, Springer, 2011.

\bibitem{Cvetkovic}D. Cvetkovi\'c, P. Rowlinson, S. Simi\'c, An Introduction to the Theory of  Graph Spectra, Cambridge University Press, Cambridge, 2010.

\bibitem{gm} E. Ghorbani, A. Majidi, Signed graphs with maximal index, arXiv: 2101.01503 (2021).

\bibitem {Koled} T. Koledin, Z. Stani\'c, Connected signed graphs of fixed order, size, and number of negative edges with maximal index, Linear Multilinear Algebra 65 (2017) 2187-2198.


\bibitem{sta1} Z. Stani\'c, Bounding the largest eigenvalue of signed graphs, Linear Algebra Appl. 573 (2019) 80--89.

\bibitem{sta2} Z. Stani\'c, Perturbations in a signed graph and its index, Discuss. Math. Graph Theory 38 (2018) 841--852.


\bibitem{Zaslavsky2}T. Zaslavsky, Matrices in the theory of signed simple graphs, Advances in Discrete Mathematics and Applications, Mysore, 2008, Ramanujan Math. Soc. Lect. Notes Ser., vol. 13, Ramanujan Math. Soc., 2010, pp. 207--229. arXiv: 1303.30837 (2013).

\end{thebibliography}
\end{document}